\documentclass[11pt]{amsart}
\usepackage{xypic}
\usepackage{tikz-cd}
\usetikzlibrary{positioning, matrix, arrows}
\usepackage{pifont}
\usepackage{tikz}
\usepackage[utf8]{inputenc}
\usepackage[T1]{fontenc}
\usepackage{amsmath,amscd,amssymb}

\newtheorem{theorem}{Theorem}[section]
\newtheorem{conjecture}[subsection]{Conjecture}
\newtheorem{lemma}[theorem]{Lemma}
\newtheorem{proposition}[theorem]{Proposition}
\newtheorem{remark}[theorem]{Remark}
\numberwithin{equation}{section}
\newtheorem{definition}[theorem]{Definition}

\input xy
\xyoption{all}

\title[On endomorphism of Fano manifold with integrable cotangent bundle]{Fano manifolds of Picard number one whose co-tangent bundle is algebraically completely integrable system and its endomorphisms}

\date{January 2022}

\begin{document}
\author{SARBESWAR PAL}  
\email{sarbeswar11@gmail.com, spal@iisertvm.ac.in}
\address{IISER Thiruvananthapuram, Maruthamala P. O., Kerala 695551}

\begin{abstract}
 Let $X$ be a projective Fano manifold of Picard number one, different from the projective space. There is a folklore conjecture that any non-constant endomorphism of $X$ is an isomorphism. In the first half of  this article, we will prove the folklore conjecture when the co-tangent bundle of $X$ is algebraically completely integrable system and the tangent bundle of $X$ is not nef. In the second half of the article, we will give examples of a collection of projective Fano manifolds of Picard rank one (different from the moduli space of vector bundles on algebraic curves) whose co-tangent bundles are algebraically completely integrable system. 
  As applications of our main theorem and examples, in fact give alternative proofs of  three major results appeared in three different articles \cite{HR, HY, SZ}.

\end{abstract}

\maketitle

\section{Introduction}
 Let $X$ be a smooth projective Fano variety of Picard number one over the field of complex numbers. \\
There is  a long-standing folklore conjecture. 
\begin{conjecture}\label{Conj1}
  If $X$ is different from the projective space, then any non-constant endomorphism $X \to X$ must be bijective.
\end{conjecture}
 The conjecture  was proved
for homogeneous spaces in \cite{PS}, for hypersurfaces of the projective space in \cite{Be}
and for Fano manifolds containing a rational curve with trivial normal bundle in
\cite{HM}; the last work solves the Conjecture in case $\text{dim}(X) = 3$. Very recently Shao, and Zhong \cite{SZ} have proved the conjecture for smooth intersection of two quadrics in $\mathbb{P}^n, n \ge 5$. The conjecture also known to be true for the moduli space of vector bundles on algebraic curves \cite{HR}.
The purpose of the article is twofold. First is to prove the above conjecture with the further assumption that $T^*X$ is algebraically completely integrable system and $TX$ is not nef. \\
More precisely, we will show the following:
\begin{theorem}\label{IT2}
Let $X$ be a Fano manifold of Picard number one, which is different from the projective space. Further, assume that $T^*X$ is algebraically completely integrable and $TX$ is not nef. Then any non-constant endomorphism is an automorphism.
\end{theorem}
 
Unfortunately, there are not many examples of such Fano manifolds known. The known example is  the fine moduli space of stable vector bundles over algebraic curves. Very recently Beauville,  Etesse,  Höring, Liu and  Voisin in \cite{BV} have shown  that the co-tangent bundle of a smooth intersection of two quadrics in $\mathbb{P}^n, n \ge 4$ is algebraically completely integrable system.   The case $n=4$ has been studied by  Kim and  Lee \cite{HY}.\\
Our next goal is to give a collection of new examples of Fano manifolds of Picard number one whose co-tangent bundles are algebraically completely integrable system. In fact, we will prove the following theorem.
\begin{theorem}\label{IT1}
    Let $X$ be a smooth intersection of two quadrics in $\mathbb{P}^{2g}, g \ge 2$, then the co-tangent bundle of the variety of $(g-2)-$ planes in $X$ is algebraically completely integrable system.
\end{theorem}
 Note that the fine moduli space of vector bundles on algebraic curves and a smooth intersection of two quadrics in $\mathbb{P}^n, n \ge 5$ are projective  Fano manifold with Picard number one. 
  Since their co-tangent bundles are  algebraically completely integrable system, the Theorem \ref{IT2}, implies that the conjecture holds true for these cases.  Also, if $g=2$, then the Theorem \ref{IT1} implies that the co-tangent bundle of a smooth intersection of two quadrics in $\mathbb{P}^4$ is algebraically completely integrable system.\\
Thus, apart from the general theorem \ref{IT2} and the new collection of examples in Theorem \ref{IT1}, our article, also unifies three major results appeared in three different articles, namely \cite{HR, HY, SZ}.\\

Our main idea to prove the conjecture \ref{Conj1} is to use Hwang-Nakayama's criterion. 
 Recall that for a surjective endomorphism  $f:X\to X$,  
a reduced divisor $D$ is called a completely invariant divisor if $(f^{*}D)_{_{\text{red}}}=D$. 
Then Hwang and Nakayama have given the following criterion: 
\begin{theorem}[Hwang-Nakayama]\label{hn}
Let $X$ be a Fano manifold of Picard number one different from the
projective space. If an endomorphism of $X$ is \'{e}tale outside a completely invariant
divisor, then it is bijective.
\end{theorem}
Recall that if $T^*X$ is algebraically completely integrable system, then there is a morphism $T^*X \to \mathbb{C}^n$, called moment map whose fibers are Lagrangian and a general fiber is isomorphic to an open subset of an abelian variety. We define a point in $X$ as wobbly if the restriction  of the moment to the co-tangent space of the point fails to be quasi-finite. The term “wobbly" was introduced
in the paper \cite{DP} for the moduli space of vector bundles on algebraic curves in different context and in different guise. The locus of wobbly points in the moduli space of vector bundles on curves has been studied in \cite{P, PP, P1}.
In the present scenario, we will show that a point is wobbly if and only if there is a non-free rational curve in $X$ through it. Using this classification, we will show that for any non-constant endomorphism of $X$, the locus of wobbly points remains completely invariant, and it is \'{e}tale outside the wobbly locus. This forces the wobbly locus to be a divisor, or the morphism is unramified. Since $X$ is simply connected, the morphisms of degree $>1$ can not be unramified. Thus, if the degree of the morphism is bigger than $1$, then  the wobbly locus is a divisor and using \ref{hn} we conclude the theorem.\\

 {\bf Organization of the paper}:\\
 In section \ref{S1}, we will recall some basic facts which we need in the subsequent sections.
 
In section \ref{S2} we will formally define the notion of wobbly type points in $X$ and  we will study some of its properties. In fact, we will show that a point $p \in X$ is wobbly if and only if there is a non-free minimal rational curve containing it. 

In section \ref{S3}, we will show that the wobbly locus is completely invariant for any non-constant endomorphism of $X$.
Finally, in section \ref{S4} we will give some examples of Fano manifold of Picard number one whose co-tangent bundle
is algebraically completely integrable system. In fact, we prove the Theorem \ref{IT1}.

\section{Preliminaries}\label{S1}
In this subsection we will recall some basic definitions and facts about Fano manifold which are mostly taken from \cite{Kol} and \cite{H1}.
Let $X$ be a Fano manifold of dimension $n$. A parametrized rational curve in $X$ is a morphism $\mathbb{P}^1 \to X$ which is birational over its image. We will not distinguish parametrized rational curves from its image $f(\mathbb{P}^1)$ and we call $f$, a rational curve in $X$.
Recall that any vector bundle on $\mathbb{P}^1$ is a direct sum of line bundles. Thus given a rational curve 
$f: \mathbb{P}^1 \to X$, the pullback of the tangent bundle $TX$ splits as 
\begin{equation}\label{e1}
f^*TX= \mathcal{O}(a_1) \oplus \mathcal{O}(a_2) \oplus ...\oplus \mathcal{O}(a_n), a_1 \ge a_2\ge...\ge a_n.
\end{equation}
The rational curve $f$ is said to be free if all the integers $a_1, a_2, ..., a_n$ in the above equation
are non-negative.\\
From now onwards, we will assume $X$ to be a Fano manifold of dimension $n$ with Picard number $1$.\\
It is known that through every point of $X$, there is a rational curve in $X$. \\

\begin{remark}\label{T1}
Note that a rational curve $f: \mathbb{P}^1 \to X$ is not free if and only if $f^*T^*X$ admits a non-zero section vanishing at some point.
\end{remark}

\section{Algebraic completely integrable system and wobbly points}\label{S2}
Recall that a symplectic manifold $M$ of dimension $2n$ is said to be a completely
integrable Hamiltonian system if there exist functions $f_1, f_2, ..., f_n$ which Poisson-
commute and for which $df_1 \wedge df_2 \wedge ...\wedge df_n$  is generically nonzero.\\
The map $f: M \to \mathbb{C}^n $ defined by these functions is called the moment map. The moment map $f$
has the property that a generic fibre is a $n$-dimensional
submanifold with $n$ linearly independent commuting vector fields $X_{f_1}, X_{f_2},..., X_{f_n}$. If
the general fibers are  open subsets in an abelian variety and the vector fields are linear, then
we shall say that the system is algebraically completely integrable. \\
Note that when $M$ is the cotangent bundle of a projective manifold $X$, then the function $f_i$ can be thought as an element of $\oplus_{i=1}^{\infty}H^0(X, S^nTX)$, where $S^nTX$ denotes the $n-$th symmetric power of $TX$. A section $s_i$ of $S^iTX$, gives a function on $T^*X$ which is a homogenious polynomial of degree $i$ in each fiber. Thus, the natural $\mathbb{C}^*$-action, namely $c.(p, v)=(p, cv)$, where $p \in X$ and $v$ is a cotangent vector of $X$ at $p$, induces an action on the sections of $S^iTX$ as follows: $c.s_i(p, v)= c^is_i(p,v)$. This action makes  the moment map $\mu$,  $\mathbb{C}^*$-equivariant.

{\bf (*)}: Let $X$ be a projective Fano manifold of Picard rank one. Then $T^*X$ has a natural symplectic structure. We further assume that $T^*X$ is algebraically completely integreable Hamiltonian system.\\
Let $\mu: T^*X \to \mathbb{C}^n$ be the moment map, which is $\mathbb{C}^*$-equivariant with respect to the   $\mathbb{C}^*$-action defined above.\\

\begin{definition}(Wobbly point)
 Let $X$ be a Fano manifold satisfying $(*)$. A point $p \in X$ is said to be \it{non-wobbly} point if and only if the restriction of the moment map $\mu$ to  the co-tangent space at $p$ is quasi finite. The points in the complement of non-wobbly points in $X$ will be called as wobbly point. 
\end{definition}
\begin{remark}\label{WD}
Since the moment map $\mu$ is $\mathbb{C}^*$-equivariant and the restriction to the cotangent space at a non-wobbly point is quasi finite, the restriction map at a non-wobbly point is in fact finite. 
\end{remark}
\begin{proposition}\label{S3P1}
Let $X$ be a projective Fano manifold satisfying $(*)$. Then a point $p \in X$ is a wobbly point if and only if there is a nonzero vector $v \in T_p^*X$, such that $\mu(p, v)= 0$.  
\end{proposition} 
\begin{proof}
Let $p \in X$ be a wobbly point. Then $\mu_{\mid_{T_p^*X}}: T_p^*X \to \mathbb{C}^n$ is not quasi finite. Thus, there 
 exist a point $x \in \mathbb{C}^n$ such that $Y:=\mu^{-1}(x) \cap T_p^*X$ is of dimension at least one. If the point $x= 0$, then we are done. If $x \ne 0$, then since $\mu$ is $\mathbb{C}^*$ equivariant, $\mu^{-1}(cx) \cap T_p^*X$ is also of dimension at least one. Now consider the one dimensional closed subset $Z$ in $\mathbb{C}^n$ difined by $x$, namely $\{cx: c \in \mathbb{C}\}$. Then $\mu^{-1}(Z) \cap T_p^*X$ is a closed subset of $T_p^*X$ and the general fiber of the restriction of the moment map has dimension at least one, each fiber has dimension at least one. In particular, $\mu^{-1}(0) \cap T_p^*X$ has dimension at least one. Hence, there is a non-zero vector $v \in T_p^*X$, such that $\mu(p, v)= 0$.\\
 Conversely, let $p \in X$ be a point such that, there is a nonzero co-tangent vector $v$ at $p$ with $\mu(p, v)=0$. Since $\mu$ is $\mathbb{C}^*$-equivariant, $\mu(p, cv)= 0$ for all $c \in \mathbb{C}^*$. Thus $\mu^{-1}(0) \cap T_p^*X$ is of positive dimensional. Hence, $p$ is wobbly.
  \end{proof}
  \begin{proposition}
  The locus of wobbly points is a proper closed subset of $X$. In other words, the non-wobbly points is a dense open subset of $X$.
  \end{proposition}
  \begin{proof}
  Note that if $p \in X$ is a  wobbly point, then there is a non-zero vector $v \in T_p^*X$, such that $\mu(p, v)=0$. Since $\mu$ is $\mathbb{C}^*$ equivariant, $\mu(p, cv)=0$ for all $c \in \mathbb{C}$, thus $\mu^{-1}(0) \cap T_p^*X$ has dimension at least one. Thus, if all the points in $X$ are wobbly, then the dimension of $\mu^{-1}(0) \ge n+1$. But every fiber of $\mu$ is Lagrangian and hence the dimension of every component of each fiber is $n$, a contradiction. Thus, the locus of non-wobbly points is non-empty.  
  \end{proof}
\begin{proposition}\label{S3P2}
Let $p \in X$, such that there is a non-free rational curve $l$ passing through $p$. Then $p$ is wobbly.
\end{proposition}
\begin{proof}
Let $l$ be a non-free  rational curve in $X$ containing $p$. Let $\mu: T^*X \to \mathbb{C}^n$ be the moment map.  Let $\mu_{\mid_l}$ be the restriction of $\mu$ to ${T^*X}_{\mid_l}$. Since $l$ is non-free, ${T^*X}_{\mid_l}$ contains a line subbundle $\mathcal{O}_{\mathbb{P}^1}(m)$ with $m \ge 1$. Since $m \ge 1, \mathcal{O}_{\mathbb{P}^1}(m)$ is globally generated. Thus, there is a non-constant section $s: l\to T^*X_{\mid_l}$ which vanish at some point different from $p$. Note that since $l$ is projective, the composition of $s$ with the moment map $\mu$ is constant. Since $s$ vanish at some point of $l$, the composition map is zero.  Thus, the restriction of the moment map to the cotangent space $T_p^*X \to \mathbb{C}^n$ has positive dimensional fiber over zero for every $p \in l.$ Thus $T_p^*X$ contains a nonzero vector $v$ such that $\mu(p,v)= 0$. Hence, by proposition \ref{S3P1}, $p$ is wobbly. 
\end{proof}
\begin{remark}\label{WD1}
Note that if $C$ is any projective curve in $X$ such that the restriction of the tangent bundle $TX$ to $C$ contains a negative degree line bundle $L$ as a  direct summand such that $L^*$ is globally generated, then by the same argument as in the proof of the previous Proposition, every point in the curve is woobly point.
\end{remark}

\begin{proposition}\label{S3P3}
 Let $x \in X$ be wobbly. Then there is a non-free  curve $C$ passing through $x$.
\end{proposition}
\begin{proof}
Since $T^*X$ is algebraically completely integrable system with the moment map $\mu$, every component of $\mu^{-1}(0)$  is of dimension $n$, and the Hamiltonian vector fields $X_{f_1}, X_{f_2}, ..., X_{f_n}$ are tangent to $\mu^{-1}(0)$.  Let $X_1, X_2, ..., X_m$ be the components of $\mu^{-1}(0)$ with $X_1$ is the component given by the zero section of $T^*X$. In other words, $X_1$ is isomorphic to $X$ and the points of $X_1$ are of the form $(p, 0), p \in X$.  Note that the components of the wobbly locus are given by $X_1 \cap X_i, i \ne 1$. Note that the Hamiltonian vector fields are tangent to $U:= X_1 \setminus \cup_{i=2}^m(X_1 \cap X_i)$. Let $W$ be a component of the wobbly locus. \\
Suppose the Hamiltonian vector fields are tangents to $X_1$ at general points of $W$. Since the Hamiltonian vector fields are commuting, $X_{f_1}((p, 0)), ..., X_{f_n}((p, 0))$ are linearly independent, provided $X_{f_i}((p, 0)) \ne 0$ for all $i$. If $X_{f_i}((p, 0))= 0$ for general points $(p, 0)$, then $X_{f_i}$ is itself zero on $W$. Therefore, $X_{f_1}, ..., X_{f_n}$ are linearly independent at general points.
On the other hand, one can see that the restriction of $T(T^*X)$ at the zero section $X_1$ is isomorphic to $TX \oplus T^*X$. Since the vector fields are tangent to $X_1$,  the projection of $X_{f_i}((p, 0))$ to $T^*X$ is zero. Hence, at a general point $(p, 0)$ of $W$, these vector fields generate the tangent space of $X_1$ at $(p, 0)$, which implies that the moment map is surjective at general points of $W$, a contradiction. \\
Thus, there is a Hamiltonian vector field, say, $X_{f_1}$ which is not tangent to $X_1$ at a general point of $W$. Let $C$ be a rational curve in $W$ through a general point.  Then, for a general point $(p, 0) \in C$, the projection of $H((p, 0))$ to $T^*X$ is non-zero, which defines a non-zero section of $T^*X \vert_C$ vanishing at some point of $C$ (the points precisely where $X_{f_1}$ is tangent to $X_1$). Hence, $C$ is not free.


\end{proof}

 \section{Woobly locus is completely invariant}\label{S3}
 Let $\mathcal{W}$ be the locus of wobbly points in $X$ and $\mathcal{W}_{\text{red}}$ be its reduced part.
 \begin{proposition}\label{p2}
Let $\varphi : X \to X$ be a non-constant endomorphism. Then the divisor $\mathcal{W}_{\text{red}}$ is completely invariant under $\varphi$, that is $\varphi^{-1}(\mathcal{W}_{\text{red}})= \mathcal{W}_{\text{red}}$. 
\end{proposition}
\begin{proof}
Note that, since $X$ has Picard rank one, any non-constant endomorphism of $X$ is finite and surjective.
Let $p \in \mathcal{W}_{\text{red}}$. Then there is a  rational curve $f: \mathbb{P}^1 \to X$, such that $p \in f(\mathbb{P}^1)$ and $f^*TX$ has a direct summand of negative degree. Then $\varphi \circ f: \mathbb{P}^1 \to X$ is a rational curve. If  $\varphi \circ f: \mathbb{P}^1 \to X$ is not free then it has a negative degree line subbundle as a direct summand and hence, by remark \ref{WD1}, $\varphi(p)$ is wobbly. 
Let us assume that  $\varphi \circ f: \mathbb{P}^1 \to X$ is free.\\
Note that if $E$ is any vector bundle on $X$, then $\mu(\varphi^*E) = m\mu(E)$, where $m$ is the degree of $\varphi$ and $\mu(E)$ denotes the slope of $E$. Thus, the Harder-Narasimhan filtration of $(\varphi \circ f)^*TX$ is the pullback of the  Harder-Narasimhan filtration of $f^*TX$. Thus, the slope of each quotient in the filtration is $m$-times of the slope of the corresponding quotient of the filtration of $f^*TX$, which is a contradiction as $f^*TX$ contains a direct summand of negative degree.\\ 
Conversely, let $p$ be a non-wobbly point. If $\varphi(p) \in \mathcal{W}_{\text{red}}$, then there exist a rational curve $g:\mathbb{P}^1 \to X$ such that $\varphi(p) \in g(\mathbb{P}^1)$ and $g$ is not free. Let $C:= \varphi^{-1}(g(\mathbb{P}^1))$. Let $\tilde{C}$ be its normalization and $h: \tilde{C} \to C \to X$ be the composition of $C \to X$ with the normalization map.  Thus, we have $(\varphi \circ h)(\tilde{C})= g(\mathbb{P}^1)$.  
 Since $g$ is not free,  as in the previous case, the restriction of $TX$ to $\tilde{C}$ contains a line subbundle of negetive degree $L$. Since such line bundle $L$ is a pullback of a line bundle on a rational curve, $L^*$ is globally generated   and $p \in h(\tilde{C})$. Hence, by remark \ref{WD1}, $p$ is wobbly, a contradiction.
This completes the proof.
\end{proof}

 Let $\varphi:X \to X$ be a surjective morphism. 
The morphism $\varphi$ will induce a morphism $d\varphi^*:\varphi^*T^*X \to T^*X$ on the total space of the cotangent bundle. Note that the fiber of $\varphi^*T^*X$ at a point $p \in X$ is isomorphic to $T^*_{\varphi(p)}X$. Thus, the moment map $\mu$ induces a map, $\varphi^*T^*X \to \mathbb{C}^n$ which also we will denote by $\mu$.  The proof of the following lemma is similar as in 
from \cite[Proposition 2.1]{KP} up to some minor modification. However, for the sake of completeness, we will record a proof here. 

\begin{lemma}\label{p5}
 There exists a morphism $\psi: \mathbb{C}^n\to \mathbb{C}^n$ such that the following diagram 
 \begin{center}
\begin{equation}\label{p4}
\begin{tikzcd}
\varphi^*T^*X\arrow[r] \arrow[d] 
 & T^*X\arrow[d] \\
 \mathbb{C}^n\arrow[r]
 & \mathbb{C}^n
 \end{tikzcd}
 \end{equation}
\end{center}
is commutes. Moreover, $\psi$ is $\mathbb{C}^*$-equivariant with respect to the action defined in section \ref{S3}.
\begin{proof}
 Let $M:= \mathbb{P}(T^*X \oplus \mathcal{O}_X)$ and $\pi: M \to X$ be the projection map. Then $M$ contains $T^*X$ as open dense subset, whose complement is a divisor $D$. Consider the fiber product diagram
 \begin{center}
\begin{equation}\label{p4}
\begin{tikzcd}
M \times_X X\arrow[r] \arrow[d] 
 & M\arrow[d] \\
 X\arrow[r]
& X
 \end{tikzcd}
 \end{equation}
\end{center}
 where the left vertical and the top horizontal maps are the projection maps $p_2, p_1$ respectively, the right vertical is the projection map $\pi$ and the lower horizontal map is the endomorphism $\varphi$. Note that the $p_1^{-1}(T^*X)$ is precisely the total space of $\varphi^*T^*X$. Let $U$ be the dense open sub set consisting non-wobbly points of $X$.  Define $\tilde{D}:= D \cap \pi^{-1}(U)$, which is open in $D$. 
 Let $Y= M \setminus \tilde{D}$ and $\tilde{Y}:= p_1^{-1}(Y)$. Note that $Y= T^*X \cup (D \cap \pi^{-1}(\mathcal{W}))$. Thus, the codimensions of the complements  of $T^*X$ and $\varphi^*T^*X$ in $Y$ and $\tilde{Y}$ respectively  are at least $2$.  Therefore, by Hartogs theorem the vertical regular maps in the theorem, extend to $\tilde{Y}$ and $Y$ respectively. \\ 
 On the other hand, the restriction of the map $\tilde{Y} \to Y$ to $\varphi^*T^*X$ is the map ${d\varphi}^*$.
  Thus, we have the following diagram:
 
 \begin{center}
\begin{equation}\label{p4}
\begin{tikzcd}
\tilde{Y} \arrow[r] \arrow[d] 
 & Y\arrow[d] \\
 \mathbb{C}^n\arrow[r]
& \mathbb{C}^n
 \end{tikzcd}
 \end{equation}
\end{center}
 
 Note that by definition of non-wobbly point, for a non-wobbly point $p$ and for  $s\in \mathbb{C}^n, T^*_pX \cap \mu^{-1}(s)$ is closed  in $M \times_X X$.  Thus, for any $s \in \mathbb{C}^n$, the closure  $\overline{\mu^{-1}(s)}$ in $M \times_X X $ is contained in $ \tilde{Y}$. But $\overline{\mu^{-1}(s)}$
projective variety.  Therefore, $\mu \circ d\varphi^{*}$ is constant.
Thus, we get a morphism $\psi: \mathbb{C}^n \to \mathbb{C}^n$ such that the following diagram commutes.
\begin{center}
\begin{equation}\label{p4}
\begin{tikzcd}
\varphi^*T^*X\arrow[r] \arrow[d] 
 & T^*X\arrow[d] \\
 \mathbb{C}^n\arrow[r]
& \mathbb{C}^n
 \end{tikzcd}
 \end{equation}
\end{center}
Clearly, $\psi:\mathbb{C}^n \to \mathbb{C}^n$ is $\mathbb{C}^*$ -equivariant since $d\varphi^*$ is $\mathbb{C}^*$-equivariant and $\mu$ is $\mathbb{C}^*$-equivariant
with respect to the induced action defined above.

\end{proof}

\end{lemma}

\begin{lemma}
 The morphism $\psi:\mathbb{C}^n\to \mathbb{C}^n$ in Lemma \ref{p5} is surjective.
 \begin{proof}
 Since $\varphi:X \to X$ is a finite morphism, it is   unramified on an open set $U\subset X$.
 Now let $V\subset X$ be the set of all non-wobbly points in $X$. Then $V$ is a non-empty open subset of $X$. Since $X$ 
  is irreducible $U\cap V\neq \varnothing$. Choose, $x\in U\cap V$. Then the moment map $\mu:T_x^*X \to \mathbb{C}^n$ is surjective. 
 On the other hand, as $\varphi$ is unramified therefore, $d\varphi^*:T_{x^{'}}^*X \to T_{x}^*X$ is surjective where $x^{'}=\varphi(x)$.
  Therefore, from the commutative diagram 
 \begin{center}
\begin{equation}\label{p4}
\begin{tikzcd}
T_{x^{'}}^*X \arrow[r] \arrow[d] 
 & T_{x}^*X \arrow[d] \\
 \mathbb{C}^n\arrow[r]
 & \mathbb{C}^n
 \end{tikzcd}
 \end{equation}
\end{center}
it follows that $\psi$ is surjective.
\end{proof}
\end{lemma}

\begin{proposition}\label{p6}
The morphism $\varphi$ is unramified outside $\mathcal{W}_{\text{red}}$.
\end{proposition}
\begin{proof}
We need to show that $\varphi_{\mid_{X \setminus \mathcal{W}_{\text{red}}}}:X \setminus \mathcal{W}_{\text{red}} \to X \setminus \mathcal{W}_{\text{red}}$ is unramified. 
Note that a finite morphism $g:Z\to Z$ of a smooth projective variety over complex numbers is unramified if and only if 
$T_{_z}Z\to T_{_{g(z)}}Z$ is injective, equivalently $T_{_{g(z)}}^*Z\to T_{_{z}}^*Z$ is surjective.

Let $x \in X \setminus \mathcal{W}_{\text{red}}$. Then $x$ and $\varphi(x)$ are both non-wobbly, as $Y_{\text{red}}$ is a completely invariant divisor.
Thus, both the vertical maps in the following commutative diagram of morphisms
\begin{center}
\begin{equation}\label{p8}
\begin{tikzcd}
T_{\varphi(x)}^*X \arrow[r] \arrow[d] 
 & T_{x}^*X \arrow[d] \\
 \mathbb{C}^n \arrow[r]
 & \mathbb{C}^n
 \end{tikzcd}
 \end{equation}
\end{center}
are surjective.
Let $K$ be the kernel of $d\varphi^*:T_{\varphi(x)}^*X \to T_{x}^*X$. If $K\neq 0$, then $\mu(K)\neq 0$ and  by the above commutative diagram 
we have, $\psi(\mu(K)) = 0$.  Since  $\psi$ is surjective and $\mathbb{C}^*$-equivariant, $\psi^{-1}(0)= 0$, which is a contradiction.
Thus $K=0$, which concludes the proof of the proposition. 
\end{proof}
{\bf Proof of Theorem \ref{IT2}}: The theorem \ref{IT2} now follows from Theorem \ref{hn}, Proposition \ref{p2} and Proposition \ref{p6}.

\section{ Examples}\label{S4}
{\bf 1}. Let $C$ be a smooth projective curve of genus $g \ge 2$ over the field of complex numbers.  Let $\mathcal{M}$ and $\mathcal{M}_H$ be the moduli space of stable vector bundles and the moduli space of Higgs bundles on $C$ of rank $r$ and fixed determinant of degree $d$ such that $r$ and $d$ are coprime. Then it is known that $\mathcal{M}$ is a Fano manifold of dimension $r^2(g-1) -g+1$ and  of Picard number one. The cotangent bundle of $\mathcal{M}$ sits inside $\mathcal{M}_H$ as an open dense subset. 
There is morphism $\mathcal{M}_H \to \mathbb{C}^{r^2(g-1) -g +1}$ known as Hitchin map. The restruction of the Hitchin map to the cotangent bundle of $\mathcal{M}$ has the property $(*)$ and $(**)$ as in Section 4. Thus, by Theorem \ref{IT2}, any non-constant endomorphism of $\mathcal{M}$ is an automorphism, which gives an alternative proof of Hwang-Ramanan's theorem \cite[Theorem 5.6]{HR}.

{\bf 2}. Let $X$ be a smooth intersection of two quadrics in $\mathbb{P}^n, n \ge 5$. Then $X$ is a projective Fano manifol of Picard number one. Furthermore, by \cite[Theorem a, b, c]{BV}, the co-tangent bundle of $X$ is algebraically completely integrable system. Thus, by Theorem \ref{IT2}, the conjecture \ref{Conj1} holds true for $X$. \\

{\bf 3}. Let $C$ be smooth projective hypereelliptic curve of genus $g\ge 2$ and $\mathcal{M}_C^s(\delta)$ be moduli space of rank $2$ stable bundles on $C$ with fixed determinant $\delta$, where $\delta$ is the trivial bundle or the line bundle $L_\omega$ associated to a Weierstrass point $\omega \in C$. It is known that $\mathcal{M}_C^s(\delta)$ is smooth quasi projective variety.
Clearly $\delta$ is $\iota-$invariant, where $\iota$ is the hyperelliptic involution on $C$. 
Then $\iota$ also induces an involution on $\mathcal{M}_C^s(\delta)$ which take $E \to \iota^*E$. Let $U:= {\mathcal{M}_C^s(\delta)}^{\iota}$ be the $\iota$-fixed locus of $\mathcal{M}_C^s(\delta)$. Then $U$ is also a smooth quasi projective variety \cite{abx}.  We will show that $T^*U$ is algebraically completely integrable system. \\
Let $\mathcal{M}_H$ be the moduli space of semistable Higgs bundles of rank $2$ and fixed determinant $\delta$ on $C$. It is known that the cotangent bundle $T^*\mathcal{M}_C^s(\delta)$ sits inside $\mathcal{M}_H$ as open dense subset. There is proper surjective morphism
\[
h: \mathcal{M}_H \to H^0(K^2),
\]
called the Hitchin map. The restriction of the Hitchin map to $T^*\mathcal{M}_C^s(\delta)$ induces a morphism, 
\begin{equation}\label{EE-1}
    h: T^*\mathcal{M}_C^s(\delta) \to H^0(K^2).
\end{equation}
We recall the relation between the spectral curve and the fiber of the Hitchin map. The details can be found in \cite{BNR}.  For a general, $s \in H^0(K^2),$ the spectral curve $C_s$ is the $2-$sheeted cover of $C$ whose branch locus is the zeros of $s$.  Let 
\[
\pi: C_s \to C
\]
 be the spectral covering map.  
The fiber of $h$ over the point $s$ is isomorphic to the Prym variety $\text{Prym}(C_s/C)$ associated to the covering $\pi$. The association is given by 
 \[
 L \to \pi_*L.
 \]
 Note that $\pi_*L$ is a vector bundle of rank $2$ on which there is a $\pi_*(\mathcal{O}_{C_s})$-module structure. The $\pi_*(\mathcal{O}_{C_s})$-module structure on $\pi_*L$ gives a morphism $\phi: \pi_L \to K \otimes \pi_*L$ such that $\text{det}(\phi)=s$. \\
 The degree of $\pi_*L= \text{ degree of }\text{Nm}(L) \otimes K^{-1} $, where $\text{Nm}: J(C_s) \to J(C)$ be the norm map between the Jacobians associated to $\pi$.\\
\begin{remark}\label{ER0}
 If $E \in U$, then $\iota$ gives an involution on the cotangent space $T^*_E\mathcal{M}_C(\delta)$ and the $\iota-$fixed locus is the co-tangent space of $U$ at $E$ \cite{JD}.
\end{remark}
\begin{proposition}\label{EP1}
Let $\delta$ be the trivial line bundle on $C$. Then $T^*U$ is completely integrable Hamiltonian system. 
\end{proposition}
\begin{proof}
Note that, $K$, the canonical bundle of $C$ is $\iota-$fixed. Thus $\iota$ acts on $H^0(K_C^2)$. Let $W:= {H^0(K_C^2)}^{\iota}$ be its fixed locus.  
 Also note that if $E \in U$ and $v \in T^*_E\mathcal{M}_C^s(\delta)$ is a $\iota$-fixed cotangent vector  , then $h(E, v) \in W$. By remark \ref{ER0}, $T^*_EU, E \in U$ is isomorphic to the $\iota-$fixed locus ${T^*_E\mathcal{M}_C^s(\delta)}^{\iota}$. Thus, the morphism $h$ in \ref{EE-1} induces a map $h:T^*U \to W$.\\
Note that $T^*U$ is a symplectic manifold with the induced symplectic structure on $T^*\mathcal{M}_C^s(\delta)$.  
Since $T^*\mathcal{M}_C^s(\delta)$ is a completely integrable Hamiltonian system and the Hitchin map is the moment map, there is a basis of $H^0(C, K^2)$ which gives a set of Poisson commuting functions on $T^*\mathcal{M}_C^s(\delta)$ and the functions are generically independent. Therefore, there is a basis of $W$, which gives Poisson commuting functions on $T^*U$ with respect to the induced Poisson structure on $T^*U$. The generically non-degenerateness of the functions in the basis follows from the fact that $h: T^*U \to W$ is surjective and $T^*U$ and $W$ have the correct dimensions.\\

\end{proof}


\begin{proposition}\label{EP2}
A general fiber of $h: T^*U \to W$ is isomorphic to an open subset of an abelian variety.
\end{proposition}
\begin{proof}

Let $s$ be a general element in $W$. Let $C_s$ be the spectral curve associated to $s$. Then one can easily see that the  genus of  $C_s$ is $4g-3$. On $C_s$ there are two involutions,  one which
changes the sheets of the cover, call it $\tau$ and another is a lift of the hyperelliptic involution on
$C$ (non-canonical), call it $\eta$. Then we will get a fixed point free $\mathbb{Z}_2$ action given by
$\tau \circ \eta$. The quotient  $\tilde{C_s}:= C_s/(\tau \circ \eta)$ is a curve of
genus $2g-1$. Let $\psi: C_s \to \tilde{C_s}$ be the quotient map. Then $\psi$ is unramified. Thus, by \cite[Proposition 3.3]{RR}, $\psi^*(J_{\tilde{C_s}}) \simeq J_{C_s}^{\tau \circ \eta}$, where $J_C$ denotes the Jacobian of the curve.  
On the other hand, we have the following commutative diagram:
 \begin{center}
\begin{equation}\label{EE0}
\begin{tikzcd}
C_s \arrow[r] \arrow[d] 
 & C \arrow[d] \\
 C_s\arrow[r]
 & C
 \end{tikzcd}
 \end{equation}
\end{center}
where the vertical maps are $\tau \circ \eta$ and $\iota$ respectively and the horizontal maps are the spectral covering maps $\pi$.\\
Note that if, $\xi \in J_{C_s}^{\tau \circ \eta}$ then from the commutative diagram \ref{EE0},  $\text{Nm}(\xi)$ is $\iota$-fixed. Since there are only finitely many $\iota-$fixed point in $J_C$ and $J_{C_s}^{\tau \circ \eta}$, is connected, $\text{Nm}(J_{C_s}^{\tau \circ \eta})$ is constant. Hence  $J_{C_s}^{\tau \circ \eta} \subset \text{Prym}(C_s/C)$. 
Let $\xi \in J(\tilde{C_s})$, then, $(\psi^* \xi)$ is $(\tau \circ \eta)-$fixed. Hence $(\tau \circ \eta)_*\psi^*\xi= \psi^*\xi$. Now, from the  commutative diagram \ref{EE0}, we have $\iota_*(\pi_*(\psi^* \xi))= \pi_*{(\tau \circ \eta)}_*(\psi^* \xi)$. Now, since, ${(\tau \circ \eta)}_*(\psi^* \xi)= \psi^* \xi$, we have $\iota_*(\pi_*(\psi^* \xi))= \pi_*(\psi^* \xi)$.
Thus, $\pi_*(\psi^* \xi)$ is an $\iota$-invariant vector bundle of rank $2$ on $C$ and for a general $\xi, \pi_*(\psi^*\xi)$ is stable.  Then the proposition follows from the following observation. 
\end{proof}

{\bf Observation O(1)}:\\
Let $s$ be a $\iota-$fixed section of $K^2$ and $\pi: C_s \to C$ be the associated spectral cover of $C$. Then $h^{-1}(s)$ is isomorphic to $P:=\text{Prym}(C_s/C)$. We also have the forgetful map, $\pi_*: P \to \mathcal{M}_C^s(\delta)$ which is known to be finite.
Let $\tilde{P}:= P \cap T^*\mathcal{M}_C^s(\delta)$. Note that $U$ is closed in $\mathcal{M}_C^s(\delta)$, thus $\pi_*^{-1}(U)$ is closed in $\tilde{P}$.\\ Let us assume that, there is a closed subset $A \subset \tilde{P}$ of dimension $= \text{dim}(U)$, such that $\pi_*(A)$ intersect $U$ in an open set. Then $A= \pi_*^{-1}(U)$.\\
Note that $h^{-1}(s) \cap T^*U \subset \pi_*^{-1}(U)$. By Proposition \ref{EP1},  $\text{dim}(h^{-1}(s) \cap T^*U)= \text{dim}(U)= \text{dim}(\pi_*^{-1}(U)$. Thus $\pi_*^{-1}(U)= A= h^{-1}(s) \cap T^*U$. Therefore, if $\xi \in A$, then the Higgs field $\phi: \pi_*\xi \to \pi_*\xi \otimes K$ is $\iota-$fixed. 

By Proposition \ref{EP1} and \ref{EP2}, we have the following theorem.
\begin{theorem}\label{ET1}
  $T^*U$  is algebraically completely integrable system. 
\end{theorem}
Let us consider the co-prime case. 
\begin{proposition}\label{EP0}
Let $\delta= L_\omega$. Let $W(\omega)$ be the subspace of, $W$ consisting the sections vanishing at $\omega$. Then $h^{-1}(W(\omega) \cap T^*U = T^*U$.  
\end{proposition}
\begin{proof}
Let us consider a general $\iota-$fixed section $s$ vanishing at $\omega$.  Since the $\iota-$fixed sections of $K^2$ are pull back of sections of a line bundle of degree $2g-2$ on $\mathbb{P}^1, s$ vanishes at $\omega$ with multiplicity $2$. Therefore, the spectral curve $C_s$ is singular with a node at $x$, where $x= \pi^{-1}(\omega)$. \\
 Let $\nu: C_s^\nu \to C_s$ be the normalization of, $C_s$
  Then we have the following commutative diagram:
  \begin{center}
  \begin{tikzcd}[column sep=small]
C_s^\nu \arrow{r}{\nu}  \arrow{rd}{\pi^\nu} 
  & C_s \arrow{d}{\pi} \\
    & C.
\end{tikzcd}
\end{center}
Let $\nu^{-1}(x)= \{x_1, x_2\}$ and let $\nu^*: J_{C_s} \to J_{C_s^\nu}$ be the pullback map. 
The pullback map $\nu^*$ in fact induces a map from $\text{Prym}(C_s/C) \to \text{Prym}(C_c^\nu/C)$.\\
Let $\tau$ be the involution on $C_s^\nu$ which changes the sheets. There is another involution,  $\eta$ in $C_s^\nu$  which is a lift (non-canonical) of the hyperelliptic involution. Then the involution $\tau \circ \eta$ has two fixed points, namely $x_1$ and $x_2$.
Then $\tau \circ \eta$ also induces an involution $\sigma$ on $C_s$ such that 
the  following  diagrams commute: 
\begin{center}
  \begin{equation}\label{EE-2}
 \begin{tikzcd}
 C_s^\nu  \arrow{r}{\pi^\nu}  \arrow[swap]{d}{\tau \circ \eta} 
   & C \arrow{d}{\iota} \\
   C^\nu_s \arrow{r}{\pi^\nu}
   & C
  \end{tikzcd}
  \end{equation}
 \end{center}

\begin{center}
  \begin{equation}\label{EE-3}
 \begin{tikzcd}
 C_s^\nu  \arrow{r}{\nu}  \arrow[swap]{d}{\tau \circ \eta} 
   & C_s \arrow{r}{\pi} \arrow{d}{\sigma} & C \arrow{d}{\iota} \\
   C^\nu_s \arrow{r}{\nu}
   & C_s \arrow{r}{\pi} & C.
  \end{tikzcd}
  \end{equation}
 \end{center}

Let
\[
f: C_s^\nu \to \tilde{C_s^\nu}= C_s^\nu/(\tau \circ \eta) 
\]
and 
\[
g: C_s \to \tilde{C_s}= C_s/(\sigma) 
\]
be the quotient maps. 
Then we have the following commutative diagram:
\begin{center}
  \begin{equation}\label{EE-4}
 \begin{tikzcd}
 C_s^\nu  \arrow{r}{\nu}  \arrow[swap]{d}{f} 
   & C_s \arrow{d}{g} \\
   \tilde{C^\nu_s} \arrow{r}{\tilde{\nu}}
   & \tilde{C}_s
  \end{tikzcd}
  \end{equation}
 \end{center}
By \cite[Proposition 3.1]{RR}, $J_{C_c^\nu}^{\tau \circ \eta} = f^*(J_{\tilde{C_s^\nu}}) + P[2]$, where $P[2]$ denotes the order two elements in $\text{Prym}(C_s^\nu/\tilde{C_s^\nu})$. Since $f$ is ramified at two points, again by \cite[P.11]{RR}, we have $P[2]=f^*J_{\tilde{C_s^\nu}}[2]$. Thus $J_{C_c^\nu}^{\tau \circ \eta} = f^*(J_{\tilde{C_s^\nu}})$. Similarly, $J_{C_s}^{\sigma}= g^*(J_{\tilde{C_s}})$ and they are contained in $\text{Prym}(C_s^\nu/C)$ and, $\text{Prym}(C_s/C)$ respectively. From the commutative diagram \ref{EE-4}, we have $f^*\tilde{\nu}^*(J_{\tilde{C_s}})= \nu^*g^*(J_{\tilde{C_s}})$. 
 Let  $A:= \{\xi \in g^*(J_{\tilde{C_s}}),  \text{ such that } \pi_*\xi \text{ stable }\}$. Then $A$ is a non-empty open subset. Note that the dimension of $A= 2g-1= \text{dim}(U)$, where determinant of the bundles in $U$ are trivial. Hence, by the observation $O(1)$, we have $A= h^{-1} \cap T^*U$. \\
 Let $L \in J_{C_s^\nu}^{\tau \circ \eta}$ and $\xi \in (\nu^*)^{-1}(L)$. Then we have the exact sequence
\[
0 \to \xi \to \nu_*\nu^* \xi \to \mathbb{C}_x \to 0.
\]
After taking the direct image we get,
\[
0 \to \pi_*\xi \to \pi_*(\nu_*\nu^* \xi) \to \mathbb{C}_{\omega} \to 0.
\]
Since $\pi \circ \nu= \pi^\nu,\pi_*(\nu_*\nu^* \xi)= \pi^\nu_*L$, which gives $\text{det}(\pi_*\xi)= \text{det}(\pi^\nu_*L) \otimes {L_{\omega}}^{-1}$. Since the Higgs field, associated to $\xi$, is $\iota-$fixed, the Higgs field associated to $L$ is also $\iota-$fixed. 
Thus, if $\pi_*^\nu L)$ is stable, then it gives a point in $T^*U$, where the determinant of the bundles in $U$ are $L_\omega$. The dimension of $J_{C_s^\nu}^{\tau \circ \eta}$ is $2g-2$ and, for a general element $L \in J_{C_s^\nu}^{\tau \circ \eta}$, is stable, we have $\pi_*^{-1}(U)= h^{-1}(s) \cap T^*U$. \\
Let $W(\omega)$ be the subspace of, $W$ consisting the sections which vanish at $\omega$. Then $h^{-1}(W(\omega)) \cap T^*U$ is a closed subset of $T^*U$ and has the same dimension as that of $T^*U, h^{-1}(W(\omega)) \cap T^*U= T^*U$.
\end{proof}

\begin{proposition}\label{EP-1}
  Let $\delta =L_\omega$. Then $T^*U$ is again an algebraically  completely integrable  system.
\end{proposition}
\begin{proof}
By proposition \ref{EP0}, the restriction of the Hitchin map to $T^*U$, induces  a map $h: T^*U \to W(\omega)$.  Note that $U$ and $W(\omega)$ have the same dimension and $h$ is surjective. Therefore, as in the trivial determinant case, we have $T^*U$ is completely integrable Hamiltonian system.
From the proof of the proposition \ref{EP0}, it follows that for a general element $s \in W(\omega)$, the fiber $h^{-1}(s)$ is isomorphic to an open subset of an abelian variety. Therefore, $T^*U$ is algebraically completely integrable system.
\end{proof}

\begin{theorem}
Let $\delta$ be the trivial line bundle on $C$. If there exist a smooth compactification $X$ of $U$ such that co-dimension of $X \setminus U$ is at least, $2$ then  $T^*X$ is algebraically completely integrable system.
\end{theorem}
\begin{proof}
 Since $U$ is isomorphic to an open subset of $X, T^*U$ is also isomorphic to an open subset of $T^*X$. As $X \setminus U$ has co-dimension at least $2, T^*X \setminus T^*U$ also has co-dimension at least $2$. Thus, the functions defining  $h$ in the proof of the Proposition \ref{EP1} extends to, $T^*X$ which are Poisson commuting and generically independent. By Proposition \ref{EP2}, a general fiber of the map $h$ is isomorphic to an open subset of an abelian variety. Hence, $T^*X$ is algebraically completely integrable system.    
\end{proof}
\begin{remark}
I feel that $U$ will always have a smooth compactification $X$ with the property as in the previous theorem, where $X$ is isomorphic to a smooth intersection of two quadrics in $\mathbb{P}^{2g+1}$.
\end{remark}
\begin{theorem}
If $ \delta = L_\omega$, then $U$ is isomorphic to the variety of $(g-2)-$planes in a smooth intersection of two quadrics in $\mathbb{P}^{2g}$. In particular, it has dimension $2g-2$.
\end{theorem}
\begin{proof}
Let $\xi=L_\omega^g$ be the line bundle of degree $g$ on $C$. Let $E$ be a stable vector bundle with fixed determinant $\delta= L_\omega$. Let $E_\xi:= E \otimes \xi$. Let $V:= \sum {\text{det}(E_\xi)}_\omega$, where the sum is taken over the Weierstrass points $\omega \in C$. Then, by \cite{DR}, $\mathcal{M}_C^s(\delta)$ is isomorphic to the variety of $(g-2)-$planes in a smooth intersection of two quadrics in $\mathbb{P}^{2g+1}= \mathbb{P}(V)$.\\
Note that the hyperelliptic involution acts on the fiber at $\omega$ of $L_\omega$ as $-\text{Id}$. This gives an action on the vector space $V$ which changes the sign at the co-ordinate corresponding to the fiber at $\omega$. The fixed locus of this action is thus a hyperplane $H$ in $\mathbb{P}(V)$. Since $\mathcal{M}_C^s(\delta)$ is isomorphic to the $(g-2)-$planes in $X, U$ is isomorphic to the $(g-2)-$planes in $X \cap H$. \\
\end{proof}
Let $X$ be a smooth intersection of two quadrics in $\mathbb{P}^{2g}$. Let $F_{g-2}(X)$ be  variety  of $(g-2)-$planes in $X$. It is known that, if $g \ge 3$, then $F_{g-2}(X)$ is a projective Fano manifold with Picard number one  \cite[Proposition 3.1]{DM}. Then we have the following theorem:
\begin{theorem}\label{ET-2}
 The co-tangent bundle of the  projective Fano manifold  $F_{g-2}(X)$, is algebraically completely integrable system and if $g \ge 3$ then any non-constant endomorphism  of $F_{g-2}(X)$ is an automorphism.
\end{theorem}
\begin{proof}
    Note that $F_{g-2}(X)$ is isomorphic to $U$. By Proposition  \ref{EP2}, $T^*U$ is algebraically completely integrable system, which concludes the theorem. 
\end{proof}

\begin{remark}
Note that if $g=2$, the $F_0(X)$ is the intersection of two quadrics in $\mathbb{P}^4$.  Hence, Theorem \ref{ET-2}, gives an alternative proof of the theorem of  Kim and Lee \cite{HY} 
\end{remark}

 {\it Acknowledgement:} I would like to thank to Dr. Suratno Basu and  Prof. A. J Parameswaran for many useful discussions, comments and suggestions. I also would like to thank
 Prof. A. Beauville for pointing out the error in the previous version of the article and for many important discussions, suggestions.


\begin{thebibliography}{111}
 
\bibitem{Be} Beauville, A.: Endomorphisms of hypersurfaces and other manifolds, Intern. Math. Res.
Notices 2001 no. 1, 53--58.
\bibitem{BNR}  Beauville, Arnaud; Narasimhan, M. S.; Ramanan, S. Spectral curves and the generalised theta divisor. J. Reine Angew. Math. 398 (1989), 169–179. 
\bibitem{BV} A. Beauville, A. Etesse, A. Höring, J. Liu, C. Voisin. Symmetric tensors on the intersection of two quadrics and Lagrangian fibration, https://arxiv.org/pdf/2304.10919.pdf
989--1004.
\bibitem{DM} O. Debarre and L. Manivel, Sur la vari \'{e}t \'{e} des espaces lin \'{e}aires contenus dans une intersection compl\'{e}te, Math. Ann. 312 (1998), 549–574.
\bibitem{HM} Hwang, J.-M. and Mok, N.: Finite morphisms onto Fano manifolds of Picard number 1
which have rational curves with trivial normal bundles, J. Alg. Geom. 12 (2003), 627--651.
\bibitem{PS} Paranjape, K. H. and Srinivas, V.: Self maps of homogeneous spaces, Invent. Math. 98
(1989), 425--444.
\bibitem{Kol} Koll\'{a}r, J.: Rational curves on algebraic varieties, Erg. d. Math. 3 Folge 32, Springer
Verlag 1996.
\bibitem{DP} R. Donagi, T. Pantev, Geometric Langlands and non-abelian Hodge theory, Surveys in differential geometry,  Vol.  XIII.  Geometry,  analysis,  and  algebraic  geometry:   forty  years  of  the  Journal  of Differential Geometry, 85–116, International Press, Somerville, MA, 2009.

\bibitem{P} S. Pal, Locus of non-very stable bundles and its geometry, Bulletin des Sciences Mathématiques 141 (2017), 747-765.
\bibitem{PP} Pal, Sarbeswar; Pauly, Christian The wobbly divisors of the moduli space of rank-2 vector bundles. Adv. Geom. 21 (2021), no. 4, 473--482.
\bibitem{P1} Pal, Sarbeswar On a Drinfeld's Conjecture, pre-print, https://arxiv.org/pdf/2202.11874.pdf.
\bibitem{HR}  Hwang, Jun-Muk; Ramanan, S. Hecke curves and Hitchin discriminant. Ann. Sci. \'{E}cole Norm. Sup. (4) 37 (2004), no. 5, 801–817. 
\bibitem{H1} Hwang, Jun-Muk Geometry of minimal rational curves on Fano manifolds. School on Vanishing Theorems and Effective Results in Algebraic Geometry (Trieste, 2000), 335--393, ICTP Lect. Notes, 6, Abdus Salam Int. Cent. Theoret. Phys., Trieste, 2001.
\bibitem{HN}  Hwang, Jun-Muk; Nakayama, Noboru On endomorphisms of Fano manifolds of Picard number one. Pure Appl. Math. Q. 7 (2011), no. 4, Special Issue: In memory of Eckart Viehweg, 1407–1426.
\bibitem{HY}  Hosung Kim, Yongnam Lee Lagrangian fibration structure on the cotangent bundle of a del Pezzo surface of degree 4. https://arxiv.org/pdf/2210.01317.pdf
Ann. 331(2005), no. 4, 925--937.
\bibitem{KP} Kouvidakis, A. and Pantev, T.: The automorphism group of the moduli space of semi-
stable bundles. Math. Annalen 302 (1995) 225-268.
\bibitem{abx} abx (https://mathoverflow.net/users/40297/abx), Smoothness of fix point components of finite group action on smooth variety, URL (version: 2014-06-10): https://mathoverflow.net/q/171529


\bibitem{RR}  Recillas, Sev\'{i}n; Rodr\'{i}guez, Rub\'{i} E. Prym varieties and fourfold covers, https://arxiv.org/pdf/math/0303155.pdf
\bibitem{DR} U.V. Desale, S. Ramanan, Classiﬁcation of vector bundles of rank 2 on hyperelliptic curves, Invent.
Math. 38 (1976/1977) 161–185.
Wiley and Sons Inc., New York, 1994, reprint of the 1978 original.
\bibitem{Nad} Alan Mi hael Nadel. The boundedness of degree of Fano varieties with Picard number
one. J. Amer. Math. So ., 4(4):681-692, 1991.
\bibitem{SZ} Feng Shao, Guolei Zhong. Boundedness of finite morphisms onto Fano manifolds with large Fano index. https://arxiv.org/abs/2211.16380
\bibitem{JD} Jason DeVito (https://math.stackexchange.com/users/331/jason-devito), Tangent space to fixed point manifold and fixed point set of tangent space, URL (version: 2017-10-28): https://math.stackexchange.com/q/2494047
\end{thebibliography}
 \end{document}